\documentclass[reqno,11pt]{amsart}
\usepackage{amsmath,amsthm,amssymb}

\date{}

\newtheorem{theorem}{Theorem}[section]
\newtheorem{lemma}[theorem]{Lemma}
\newtheorem{corollary}[theorem]{Corollary}
\newtheorem{proposition}[theorem]{Proposition}
\newtheorem{example}[theorem]{Example}
\newtheorem{remark}[theorem]{Remark}

\newcommand{\Exp}{{\rm I\!E}}

\begin{document}

\title{Hessian metrics, $CD(K,N)$-spaces, and optimal transportation of log-concave measures}

\begin{abstract}
We study the optimal transportation mapping
$\nabla \Phi : \mathbb{R}^d \mapsto \mathbb{R}^d$ pushing forward a  probability measure $\mu = e^{-V} \ dx$  onto another probability measure $\nu = e^{-W} \ dx$.
Following a  classical 
approach of E.~Calabi we introduce the Riemannian metric $g = D^2 \Phi$  on $\mathbb{R}^d$ and study  spectral properties of the metric-measure space
$M=(\mathbb{R}^d, g, \mu)$. 
  We prove, in particular, that $M$ admits a non-negative Bakry--{\'E}mery tensor  provided both $V$ and $W$ are convex.
If the target measure $\nu$ is the Lebesgue measure on a convex set $\Omega$ and $\mu$ is log-concave we prove that $M$ is a  $CD(K,N)$ space.
Applications  of these results include some global  dimension-free a priori estimates  of $\| D^2 \Phi\|$. With the help of comparison techniques on Riemannian manifolds and probabilistic concentration arguments
we proof some diameter estimates for $M$.
\end{abstract}


\author{Alexander V. Kolesnikov}

\maketitle

\section{Introduction}

This paper is motivated by the following problem. Given two probability measures  $\mu = e^{-V} \ dx$ and $\nu = e^{-W} \ dx$ on $\mathbb{R}^d$
let us consider the optimal transportation mapping $T=\nabla \Phi$ of $\mu$ onto $\nu$
and the associated Monge--Amp{\`e}re equation
\begin{equation}
\label{ma}
e^{-V} = e^{-W(\nabla \Phi)} \det D^2 \Phi.
\end{equation}
We are interested in efficient estimates of the Lipschitz constant
$\sup_{\mathbb{R}^d} \| D^2 \Phi\|$ or the  integral Lipschitz constant 
$ \bigl(\int_{\mathbb{R}^d} \| D^2 \Phi\|^{p}\ d \mu\bigr)^{\frac{1}{p}}$ (here $\|\cdot\|$ is the operator norm) for some $p \ge 1$.

This problem has different aspects. From  the regularity theory viewpoint the ''best'' estimate provides the  highest regularity level for $\Phi$ 
for any given regularity assumptions on $V$ and $W$. Classical regularity results  provide  $\Phi \in C^{2,\alpha}$ under  assumption of the H{\"o}lder continuity of  $V$ and $W$. 
The results of this type
are usually available only in finite dimensions and involve constants which are very hard to control. For an account in the regularity theory of the optimal transportation and the
 Monge--Amp{\`e}re equation
(flat case) see  \cite{Guit}, \cite{TrudW}   (see also \cite{Bak}, \cite{Pog}, \cite{GT},  \cite{CafCab}, \cite{Kryl2},  \cite{Caf-90}, \cite{CNS}).

Our motivation partially comes from  the optimal transportation theory on the infinite-dimensional spaces, in particular,
on the Wiener space (\cite{BoKo2005}, \cite{BoKo2011}, \cite{FU1}, \cite{Kol04}). Note that the finite-dimensional regularity  techniques can not be applied
 here and the general regularity problem for optimal transportation on the Wiener space is open.  Some partial results see in \cite{BoKo2011}.

Another type of problems  is studied in convex geometry. Given the target measure $\nu$ (a typical example: $\nu$ is the normalized Lebesgue measure on a convex set)
find a "nice" source measure  $\mu$ with known spectral properties (say, known Poincar{\'e} constant)  and small $L^p(\mu)$-norm of $\|D^2 \Phi\|$. A classical example is given by a  Caffarelli's contraction theorem.
 According to this result  every optimal transportation mapping $\nabla \Phi$ pushing forward the standard Gaussian measure onto  a log-concave measure $\nu$ with the uniformly convex potential 
$W$ (i.e. $D^2 W \ge K \cdot \mbox{Id}$  with $K>0$)
 is a $\frac{1}{\sqrt{K}}$ - contraction (i.e. $\|D^2 \Phi\|_{Lip} \le \frac{1}{\sqrt{K}}$).

This result implies immediately very nice analytical consequences (for instance, the isoperimetric comparison Bakry--Ledoux theorem, a probabilistic version of the L{\'e}vy-Gromov
comparison theorem). More about it see in  \cite{Koles2011}. 
Note that for many applications a dimension-free bound of the integral norm $ \bigl(\int_{\mathbb{R}^d} \| D^2 \Phi\|^{p}\ d \mu\bigr)^{\frac{1}{p}}, p \ge 1$ 
would be sufficient.
This follows  from  a recent result of Emanuel Milman (see \cite{Milman2008}) 
on equivalence of norms in the log-concave case.

Let us recall  a related  open problem.
A convex set $\Omega$ is called isotropic if $$\Exp (x_i)=0,  \ \ \Exp (x_i x_j) =\delta_{ij}$$ (here $\Exp$ means the expectation with respect to the normalized Lebesgue volume on $\Omega$).
The Poincar{\'e} constant is the minimal constant $c_p$ such that
$$
\Exp f^2 - (\Exp f)^2 \le c_p \cdot \Exp |\nabla f|^2 
$$
for any smooth $f$.
 According to the famous Kannan, Lov{\'a}sz, and Simonovits conjecture (KLS conjecture)
the Poincar{\'e} constant of any isotropic convex set is  bounded by some universal number.
This is one of the most difficult open problems naturally arising in convex geometry (KLS conjecture, slicing problem, thin-shell conjecture). 

By a  classical result of Payne and Weinberger (see \cite{PW})  
$$
c_p \le \frac{ \mbox{\rm{diam}}^2 (\Omega)}{\pi^2}
$$
for any convex $\Omega$.
Thus in view of the Caffarelli's theorem it is natural to expect that the Lipschitz constant of the optimal transportation mapping
pushing forward, say, the standard Gaussian measure $\gamma$ onto a convex set $\Omega$ is controlled at least by the diameter of $\Omega$.
However, even this turns out to be difficult to prove. According to \cite{Kol} this Lipschitz constant is controlled by 
$\sqrt{d} \ \mbox{diam}(\Omega)$. It is still not known whether  $\int \| D^2 \Phi \| \ d \gamma \le C \ \mbox{diam}(\Omega)$ for some universal $C$.

We prove in this paper that
\begin{equation}
\label{main-appl}
\int \Lambda \ d \mu - \Bigl( \int \sqrt{\Lambda} \ d \mu \Bigr)^2   \le c \ \mbox{diam}(\Omega), 
\end{equation}
where $\Lambda(x) = \| D^2 \Phi(x)\|$, $\mu=\gamma$ is the standard Gaussian measure, $c$ is  universal, and  $\nu$ is the normalized Lebesgue measure on $\Omega$.

We apply here a geometric approach developed by E.~Calabi  in \cite{Calabi}
in his study of the regularity problem for the Monge--Amp{\`e}re equation. The sharpest regularity results  for the Monge--Amp{\`e}re equation have been obtained later 
by other methods, including the Krylov-Safonov-Evans regularity theory.
Let us cite Nikolai Krylov (see \cite{Kryl2}). 
{\it ''To prove the existence of solutions of equations like (\ref{ma}) by the methods known before 1981 was no easy task. It involved finding a priori estimates for solutions and their derivatives up 
to {\bf third} order. Big part of the work is based on differentiation (\ref{ma}) three times and on certain extremely cleverly organized manipulations invented by Calabi. After 1981 the approach to 
fully nonlinear equations changed dramatically.''}
Unfortunately, the deep techniques  developed by Krylov, Safonov,  Evans, Caffarelli and others 
don't work in applications where some {\bf dimension-free} bounds needed.
These applications include convex geometry and  analysis on Wiener space. This is the main reason why we come back to the old Calabi's trick.

Let us explain the main idea of this approach. Differentiating (\ref{ma}) along a vector $e$ one obtains the following quasilinear diffusion equation
$$
\partial_e V = - L_{\Phi} \Bigl( {\partial_e} \Phi \Bigr),
$$
where
$L_{\Phi}  f  =    \mbox{div}_{\nu} \Bigl( \nabla  f ( \nabla \Psi) \Bigr) \circ (\nabla \Phi)$ and $\Psi =\Phi^*$ is the Legendre transform of $\Phi$.
This diffusion operator is a generator  of the following symmetric Dirichlet form:
$$
\mathcal{E}_{\Phi} (f,h) = 
\int \bigl< (D^2 \Phi)^{-1} \nabla f, \nabla h \bigr> \ d\mu.
$$ 
It is natural (and it was the Calabi's observation) to introduce the following Riemannian metric 
on $\mathbb{R}^d$: $g_{ij} = \Phi_{x_i x_j}$ (in the fixed initial coordinate system). The corresponding manifold will be denoted by $M$. The 
manifolds of this type are called "Hessian manifolds" and they are real analogs of the complex K{\"a}hler manifolds.
More on Hessian manifolds see in \cite{Ch-Yau}, \cite{Shima}.

The Dirichlet form  $\mathcal{E}_{\Phi}$ can be rewritten as follows:
$$
\mathcal{E}_{\Phi} (f,h) = \int \bigl<  \nabla_M f, \nabla_M h \bigr>_M \ d\mu,
$$
where $\nabla_M$ is the gradient on $M$. Note that $\mu$ is {\bf not} the Riemannian volume $\mbox{vol}_M$ measure of $M$. In fact
$$
\mathcal{E}_{\Phi} (f,h) = \int \bigl<  \nabla_M f, \nabla_M h \bigr>_M \ e^{-P} \ d  {\rm vol}_M, 
$$
where $P = \frac{1}{2} \bigl( V + W(\nabla \Phi) \bigr)$. We introduce the {\it metric-measure space }
$(\mathbb{R}^d, g,  \mu)$  denoted by the same letter $M$.

The diffusion generator  can be rewritten in geometric terms $$L_{\Phi} = \Delta_M - \nabla_M P \cdot \nabla_M.$$

We note that the inverse optimal mapping $\nabla \Psi = (\nabla \Phi)^{-1}$ defines a dual metric-measure space $M'=(\mathbb{R}^d, D^2 \Psi, \nu)$  and the mapping $x \to \nabla \Phi(x)$
is a measure preserving isometry between $M$  and $M'$.

In this paper we compute the  second "carr{\'e} du champ" operator $\Gamma_2$    which is responsible for spectral  properties of the Dirichlet form
$\mathcal{E}_{\Phi}$. Equivalently, we compute the Bakry--{\'E}mery tensor $\mbox{Ric} + D^2_M P$ of the metric-measure space $(\mathbb{R}^d, g, \mu)$. In particular, we get
\begin{theorem}
\label{BE-est}
Assume that $V$ and $W$ are convex. Then the Bakry--{\'E}mery tensor of  $M = (\mathbb{R}^d, g, \mu)$ is non-negative.
\end{theorem}

Further  we investigate the concentration properties of $M$.
To this end we establish the following elementary but useful inequality
\begin{equation}
\label{d2m<ddp}
d^2_M(x,y) \le d(x,y)  \cdot d(\nabla \Phi(x), \nabla \Phi(y)),
\end{equation}
where $d$ is the standard Euclidean distance. 

We are especially interested in a particular case when the source measure $\mu$ is log-concave and the target measure $\nu$
is the Lebesgue measure on a convex set $A$. It turns out that in this case one can establish a stronger result than Theorem \ref{BE-est}. We prove namely that $M$
belongs to the family of the so-called $CD(K,N)$ spaces.
\begin{theorem}
Assume that $\mu = e^{-V}  dx$ is a log-concave measure and $\nu$
is the Lebesgue measure on a convex set $A$. Then $M$ is a $CD(0,2d)$-space.

If, in addition, $$D^2 V \ge C \cdot \mbox{{\rm Id}}$$ with $C>0$, then $(\mathbb{R}^d,g,\mu)$ is  a $CD\bigl(\frac{C}{m},2d\bigr)$-space,
where $m = \sup_{x \in \mathbb{R}^d} \|D^2 \Phi\|$.
\end{theorem}

The metric-measure spaces satisfying $CD(K,N)$-condition are widely studied in analysis. The most significant applications of this concept  are deeply related to the optimal transportation theory (see \cite{Villani2}).
Roughly speaking, these spaces  have analytical and geometrical properties comparable to those of  the corresponding model spaces (spheres and hyperbolic spaces). The list
 of properties which can be extracted from the $CD(K,N)$ property is rather impressive. These are Sobolev and isoperimetric inequalities, diameter bounds, volume growth of balls, Laplacian 
comparison estimates etc. In particular, the general theory implies  several interesting results in our special situation.
Applying inequality (\ref{d2m<ddp}) to the case when 
$\nu$ is the normalized Lebesque measure on a convex set $\Omega$ and $\mu = \gamma$ is the standard Gaussian measure,
 we get that $M$ has a superquadratic (fourth order) concentration function. It is known that the Gaussian-type concentration together with
an appropriate $CD(K,N)$-condition implies that $M$ has a bounded diameter. More precisely, we get
\begin{theorem}
\label{diameterbound}
Let $\mu = \gamma$ be the standard Gaussian measure and $\nu$ be the Lebesgue measure  on a convex set 
with diameter $D$. Assume, in addition, that $M$ is geodesically convex.
There exists a universal constant $C>0$ such that
$$
\mbox{\rm{diam}}(M) \le C \sqrt[4]{d}  \sqrt{D}.
$$
\end{theorem}

We recall that a smooth (non-complete) Riemannian manifold $M$ is called geodesically convex if  with every two points $x,y \in M$ it 
contains a  smooth  geodesic path $\gamma : [0,T] \mapsto M$  joining $x$ and $y$: $\gamma(0)=x, \gamma(T)=y$ in such a way that $\gamma$ is the shortest way 
from $x$ to $y$. By the Hopf-Rinow theorem every complete manifold is geodesically convex.
It seems that the assumption of geodesical convexity of $M$   in Theorem \ref{diameterbound} can be omitted, but we were not able to prove this.

According to a recent result of E. Milman \cite{Milman} 
 the concentration and isoperimetric inequalities are equivalent in the case of the positive Bakry--{\'E}mery tensor.  In particular, applying our concentration estimates, we get that
 the Poincar{\'e} constant of $M$  depends on the diameter of $\Omega$ only.
Then applying techniques developed in \cite{Koles2010} we prove that if $\mu$ is standard Gaussian and $\nu$  normalized Lebesgue on a convex set, then
$$ 
\int \frac{\langle \nabla_M \Lambda, \nabla_M \Lambda \rangle_{M}}{\Lambda}  \ d \mu \le 1 .
$$
Then (\ref{main-appl}) follows from the Poincar{\'e} inequality for $(\mathbb{R}^d, g, \mu)$. We also establish certain reverse H{\"o}lder inequalities for $\Lambda$.

Applications of the Hessian structures in convex geometry  can be found in a recent paper of Bo'az Klartag and Rohen Eldan \cite{EK}.
 In particular,  it was shown in \cite{EK}  that the positive solution to the thin shell conjecture implies
the positive solution to the slicing problem.   Applications of the K{\"a}hler metrics to the Poincar{\'e}-type inequalities
and thin-shell estimates can be found in  \cite{Klartag}. It was  pointed out to the author by Klartag 
that some of the results from \cite{Klartag} can be generalized by the methods obtained in this paper.  More
on K{\"a}hler manifolds and convex sets see in  \cite{Gromov}.

The author is grateful to  Emanuel Milman, Bo'az Klartag, and Ronen Eldan for stimulating discussions during his visit to the Technion university of Haifa and the Tel-Aviv University.

This research  was carried out within ``The National Research University Higher School of Economics''
Academic Fund Program in 2012-2013, research grant No. 11-01-0175 and supported by the RFBR projects 10-01-00518, 11-01-90421-Ukr-f-a,
 and the program SFB 701 at the University of Bielefeld.

\section{Diffusion viewpoint}

We consider the optimal transportation mapping $\nabla \Phi$
pushing forward  a probability measure $\mu = e^{-V} \ dx$ onto another probability measure $\nu = e^{-W} \ dx$.  
In order to avoid unessential technicalities we assume in this section that $V$, $W$, $\Phi$, and 
$$\Psi(y) = \Phi^* (y)  = \sup _y (\langle x, y \rangle - \Phi(x))$$ are smooth functions  on the whole $\mathbb{R}^d$.

In particular, $\nabla \Phi, \nabla \Psi$ are reciprocal mappings: $\nabla \Psi \circ \nabla \Phi(x)=x$ and
the Hessians of $\Phi, \Psi$ satisfy the following identity
$$
D^2 \Phi (x) \cdot D^2 \Psi(\nabla \Phi(x)) = \mbox{Id}
$$
for every $x$. This means, in particular, that $D^2 \Phi$, $D^2 \Psi$ are always non-degenerate (positive) matrices.

By the change of variables formula
$$
V  = W(\nabla \Phi) - \log \det D^2 \Phi. 
$$
Let us differentiate this formula along the vector $e \in \mathbb{R}^d$.
Everywhere below the partial derivative of $f$ along $e$ 
$$
\frac{\partial f}{\partial e}(x) = \lim_{t \to 0} \frac{f(x + te) - f(x)}{t}
$$
will be denoted either
by $\frac{\partial f}{\partial e}$ or (for brevity)  by $f_e$. 
Thus, we use breve notations
$$
\nabla f_e, D^2 f_e
$$
for 
$$
\nabla \Bigl( \frac{\partial f}{\partial e} \Bigr), \ D^2 \Bigl( \frac{\partial f}{\partial e} \Bigr).
$$

After these agreements we can write the result of differentiating of the change of variables formula in the following form:
\begin{equation}
\label{diffusion}
V_{e} = \langle \nabla  \Phi_e, \nabla W( \nabla \Phi) \rangle - \mbox{Tr} \Bigl[ D^2 \Phi_{e} \cdot (D^2 \Phi)^{-1}\Bigr].
\end{equation}
Let us consider   the following diffusion operator:
$$
L_{\Phi}  f = \mbox{Tr} \bigl[D^2 f \cdot (D^2 \Phi)^{-1}  \bigr] -  \langle \nabla  f, \nabla W( \nabla \Phi) \rangle.
$$

\begin{lemma}
The following identity holds for any $f, h \in C^{\infty}_0(\mathbb{R}^d)$:
$$
\int \bigl< (D^2 \Phi)^{-1} \nabla f, \nabla h \bigr> \ d\mu =  -\int f \cdot L_{\Phi}  h \ d \mu = - \int h \cdot L_{\Phi}  f \ d \mu.
$$
\end{lemma}
\begin{proof}
 Let us apply the relations $\nabla \Psi = (\nabla \Phi)^{-1}$ and $D^2 \Psi = (D^2 \Phi)^{-1} \circ \nabla \Psi$. One has
$$
L_{\Phi}  f \circ \nabla \Psi =   \mbox{Tr} \bigl[ D^2 f ( \nabla \Psi) \cdot  D^2 \Psi \bigr] - \langle \nabla  f ( \nabla \Psi), \nabla W \rangle.
$$
Note that
\begin{equation}
\label{v-div}
L_{\Phi}  f (\nabla \Psi) =    \mbox{div}_{\nu} \Bigl( \nabla  f ( \nabla \Psi) \Bigr),
\end{equation}
where $\mbox{div}_{\nu} v = \mbox{div} (v) - \langle v, \nabla W \rangle$ is the divergence of the vector field $v$ with respect to  $\nu$.
 Hence
\begin{align*}
\int f \cdot L_{\Phi}  h \ d \mu &
=
\int f(\nabla \Psi) \mbox{div}_{\nu} \bigl( \nabla  h ( \nabla \Psi) \bigr)  \ d \nu
=
- \int \langle D^2 \Psi \cdot \nabla f(\nabla \Psi),  \nabla  h ( \nabla \Psi) \rangle  \ d \nu
\\&   = -
\int \bigl< (D^2 \Phi)^{-1} \nabla f, \nabla h \bigr> \ d\mu.
\end{align*}
\end{proof}
{\bf Conclusion:}  $\Phi_e$ satisfies the following quasilinear diffusion equation:
\begin{equation}
\label{MA-diffusion}
V_e = - L_{\Phi} \Phi_{e},
\end{equation}
where $L_{\Phi}$ is the generator of the Dirichlet form
$$
\mathcal{E}_{\Phi} (f,h) = 
\int \bigl< (D^2 \Phi)^{-1} \nabla f, \nabla h \bigr> \ d\mu.
$$

It is well-known that the second  "carr{\'e} du champ" operator $\Gamma_2$ introduced by D.~Bakry (see, for instance, \cite{Bakry85})
is responsible for spectral properties of the corresponding diffusion
$$
\Gamma_2(f) = \frac{1}{2} \Bigl( L_{\Phi} \Gamma_{\Phi} (f) - 2 \Gamma_{\Phi}(L_{\Phi} f, f) \Bigr),
$$
where  $\Gamma_{\Phi} f  = \bigl< (D^2 \Phi)^{-1} \nabla f, \nabla f \bigr>$.
We will compute $\Gamma_2$ in Sections 3-4.

\begin{example}
{\bf Computation of  $\Gamma_2$ in the one-dimensional case:}

Consider a  Dirichlet form
$$
\mathcal{E}(f) = \int_{\mathbb{R}} a (f')^2 \ d \mu.
$$ 
One has
$$
\Gamma_2(f) = a^2 (f')^2 + aa' f' f^{''} + \frac{1}{2} (f')^2 \bigl( (a')^2 - a a^{''} + a a' V' + 2 a^2 V^{''} \bigr).
$$
Let $a = \frac{1}{\varphi^{''}}$, where $\varphi$  solves
$
e^{-V} = e^{-W(\varphi')} \varphi^{''}.
$ 
We get from this equation that
$$
\frac{\varphi^{'''}}{\varphi^{''}} = W'(\varphi^{'}) \varphi^{''}  - V',
$$
$$
\frac{\varphi^{(4)}}{\varphi^{''}}  - \Bigl(  \frac{\varphi^{'''}}{\varphi^{''}} \Bigr)^2 = W''(\varphi^{'}) (\varphi^{''})^2   + W'(\varphi^{'}) \varphi^{'''} - V^{''}.
$$
Substituting this into the formula for  $\Gamma_2$, one can easily obtain
$$
\Gamma_2(f) = \frac{(f'')^2}{(\varphi'')^2}
- \frac{\varphi^{'''}}{(\varphi'')^3} f' f^{''}
+ \frac{1}{2} (f')^2 \Big[ \frac{(\varphi^{'''})^2}{(\varphi^{''})^4}  + \frac{V''}{(\varphi^{''})^2} + W''(\varphi') \Bigr].
$$
\end{example}

\section{Differential-geometric viewpoint. }

Recall that the Hessian of a smooth function $f$ on a Riemannian manifold $M$ is the tensor defined on a couple of vector fields $X, Y$
by the following  formula
$$
 D^2_M f (X, Y) = \langle \nabla_X \nabla_M f, Y \rangle_M.
$$ 
In coordinates
\begin{equation}
\label{hess-coor}
(D^2 f)_{ik}= \frac{\partial^2 f}{\partial x_i \partial x_k}  - \Gamma^j_{ik} \frac{\partial f}{\partial x_j},
\end{equation}
where $ \Gamma^j_{ik}$ are the corresponding Christoffel symbols.
Following  E.~Calabi  \cite{Calabi} we consider a Riemannian metric $g$ on $\mathbb{R}^d$ given by the Hessian of the function $\Phi$ (with respect to the standard Euclidean connection).  In a fixed standard orthogonal coordinate system one has
$$
g_{ij} = \Phi_{x_i x_j}.
$$ 
In the computations below we follow the standard geometric agreements: 
(Einstein summation) the expressions 
$A^i B_i, A^{ij} B_{ij}$ etc. mean that one takes a sum over repeating indexes:
$$
A^i B_i := \sum_{i=1}^d A^i B_i, \ A^{ijk} B_{ijl} := \sum_{1 \le i,j \le d} A^{ijk} B_{ijl}
$$
The inverse metric $g^{-1}$ is denoted by $g^{ij}$.

 The Riemannian  gradient can be computed as follows:
$$
(\nabla_M f)_j = g^{ij}\frac{\partial f}{\partial x_i}.
$$
Our manifold $M$ belongs to the class of the so-called Hessian manifolds (see \cite{Ch-Yau}, \cite{Shima}), which are real analogs of the K{\"a}hler manifolds
intensively studied in differential geometry.
Some of the computations below can be found in  \cite{Calabi} or \cite{Shima} but we give them for completeness of the picture.

All the objects related to $M$ (considered as a Riemannian manifold) will be written with the subscript $M$: $\nabla_M$ is the gradient, $D^2_M$ is the Hessian, and $\Delta_M$ is the Laplace-Beltrami operator on $M$.

Let us rewrite the  Dirichlet form 
$$
\mathcal{E}_{\Phi} (f,h) = 
\int \bigl< (D^2 \Phi)^{-1} \nabla f, \nabla h \bigr> \ d\mu
$$
 in  geometric terms:
$$
\mathcal{E}_{\Phi} (f,h) = \int_{M} \langle\nabla_M f, \nabla_M h \rangle_{M}  \ d \mu.
$$ 
Computation of  the Riemannian volume 
$$
\mbox{\rm vol}_M = \sqrt{\det g} \ dx = e^{\frac{1}{2} W(\nabla \Phi) - \frac{1}{2} V } \ dx
$$
 gives  another useful expression for  $\mathcal{E}_{\Phi}$:
\begin{equation}
\label{DirFormDiff}
\mathcal{E}_{\Phi} (f,h)  =  \int_{M} \langle\nabla_M f, \nabla_M h \rangle_{M}    e^{-P}  \ d   \mbox{\rm vol}_M ,
\end{equation}
where
$$
P = \frac{1}{2} \bigl( W(\nabla \Phi) +  V \bigr).
$$
In what follows we compute the Ricci tensor of $g$. The Christoffel symbol $$\Gamma^{k}_{ij} = \frac{1}{2} g^{kl}\Bigl( \frac{\partial g_{lj}}{\partial x^i}  + \frac{\partial g_{il}}{\partial x^j}  - \frac{\partial g_{ij}}{\partial x^l} \Bigr)$$
takes a simplified form
$$
\Gamma^{k}_{ij} =  \frac{1}{2} g^{kl} \Phi_{ijl},
$$
where 
$$
\Phi_{ijl} = \frac{\partial^3 \Phi}{\partial x_i \partial x_j \partial x_l}.
$$
For computing the Ricci curvature tensor we apply the following well-known formula for the Riemannian tensor with lowered indexes:
$$
R_{ijkl} = \frac{1}{2} \Bigl( \frac{\partial^2 g_{il}}{\partial {x^j} \partial {x^k}}  + \frac{\partial^2 g_{jk}}{\partial {x^i} \partial {x^l}}  - \frac{\partial^2 g_{ik}}{\partial {x^j} \partial {x^l}}  - \frac{\partial^2 g_{jl}}{\partial {x^i} \partial {x^k}} \Bigr)
+ g_{ms}\Bigl( \Gamma^{m}_{jk} \Gamma^{s}_{il}-  \Gamma^{m}_{ik} \Gamma^{s}_{jl} \Bigr). 
$$
The first part of this expression  vanishes and we get
$$
R_{ijkl} = \frac{1}{4}  g^{ms}\Bigl( \Phi_{mil}  \Phi_{sjk} -  \Phi_{mik}  \Phi_{sjl}  \Bigr) .
$$
Hence
$$
{\rm Ric}_{ik} =  \frac{1}{4}  g^{jl} g^{ms}\Bigl( \Phi_{mil}  \Phi_{sjk} -  \Phi_{mik}  \Phi_{sjl}  \Bigr) .
$$
Now we take into account the Mong{e}-Amp{\`e}re equation. Differentiating
$$
\log \det g = W(\nabla \Phi) - V
$$
we get another version of (\ref{MA-diffusion}):
$$
2\Gamma^{i}_{ik} = \frac{\partial \log \det g}{\partial x_k} =  g^{il} \Phi_{ikl} = g_{ik} \frac{\partial W}{\partial {x_i}} ( \nabla \Phi) - \frac{\partial V}{\partial x_k}.
$$
Finally we get the following expression for the Ricci tensor:
\begin{equation}
\label{ricci-V-W}
{\rm Ric}_{ik} =  \frac{1}{4}  g^{jl} g^{ms} \Phi_{mil}  \Phi_{sjk} -   \frac{1}{4}   g^{ms} \Phi_{mik}  \Bigl( g_{js} \frac{\partial W}{\partial {x_j}} ( \nabla \Phi) - \frac{\partial V}{\partial x_s} \Bigr).
\end{equation}
Note that the first part defines a  non-negative quadratic form:
$$
 \frac{1}{4}  g^{jl} g^{ms} \Phi_{mil}  \Phi_{sjk} \xi^i \xi^k  =  \frac{1}{4}  g^{jl}  g^{ms} (\Phi_{mil} \xi^i ) ( \Phi_{sjk} \xi^k)  \ge 0.
$$

\section{$\Gamma_2$-operator and geometric properties of $M$}

In this section we calculate the second carr{\'e} du champ operator $\Gamma_2$
$$
\Gamma_2(f) = \frac{1}{2} \Bigl( L_{\Phi} \Gamma_{\Phi} (f) - 2 \Gamma_{\Phi}(L_{\Phi} f, f) \Bigr).
$$
Applying  formula (\ref{DirFormDiff}) which represents the Dirichlet form $\mathcal{E}_{\Phi}$ via the energy  integral over $M$ equipped with the measure $\mu = e^{-P} \ d vol$
we rewrite the diffusion operator $L_{\Phi}$ as follows:
$$
L_{\Phi} = \Delta_M - \nabla_M \cdot \nabla_M P. 
$$
We apply here the Bochner's identity
$$
\|D^2_M f\|^2_{HS} +  \mbox{Ric}(\nabla_M f, \nabla_M f) = \frac{1}{2} \Delta_M |\nabla_M f|^2 - \langle \nabla_M f, \nabla_M \Delta_M f \rangle_M.
$$
A generalization of this formula to the metric-measure spaces (see, for instance, \cite{Villani2}) gives the following  expression for $\Gamma_2$: 
$$
\Gamma_2 (f) = \| D^2_M f \|^2_{{HS}} + \bigl( \mbox{Ric} + D^2_M P \bigr) (\nabla_M f, \nabla_M f).
$$
Recall that the quantity
$$
\mbox{\rm{R}}_{\infty,\mu} := \mbox{Ric} + D^2_M P
$$ is called the Bakry--{\'E}mery tensor. 
The notation $
\mbox{\rm{R}}_{\infty,\mu}$ will be explained in the very last section of the paper.

The Bakry--{\'E}mery tensor has been introduced
in \cite{BE}. According to a classical result of   Bakry and {\'E}mery  the
positivity of this tensor implies the log-Sobolev inequality for manifolds with measures.

Let us compute $\mbox{Ric} + D^2_M P$. Applying the coordinate expression for the Hessian  (\ref{hess-coor}) we get
$$
(D^2_M f)_{ik} 
= \frac{\partial^2 f}{\partial x_i \partial x_k}  -   \frac{1}{2} g^{jl} \Phi_{ikl} \frac{\partial f}{\partial x_j}.
$$
Consequently
\begin{align*}
(D^2_M f(\nabla \Phi))_{ik}  &
= g_{il} g_{kj} \frac{\partial^2 f}{\partial x_l \partial x_j} \circ\nabla \Phi  + \Bigl( \frac{\partial g_{ij}}{\partial x_k}  - \Gamma^s_{ik} g_{sj}\Bigr)\frac{\partial f}{\partial x_j} \circ \nabla \Phi
\\&
= g_{il} g_{kj} \frac{\partial^2 f}{\partial x_l \partial x_j} \circ\nabla \Phi  + \frac{1}{2} \Phi_{ijk}\frac{\partial f}{\partial x_j} \circ \nabla \Phi
\end{align*}

Differentiating the relation $D^2 \Phi \cdot D^2 \Psi(\nabla \Phi)=\mbox{Id}$ one can easily obtain
than for every vector $e$ the following identity holds
$$
D^2 (\partial_e \Phi) \cdot D^2 \Psi(\nabla \Phi) +  D^2 \Phi \cdot D^2 (\partial_e \Psi)(\nabla \Phi) \cdot D^2 \Phi=0.
$$
Hence
$$
 D^2 (\partial_e \Psi)(\nabla \Phi) = - (D^2 \Phi)^{-1} D^2 (\partial_e \Phi) (D^2 \Phi)^{-2}.
$$
With the help of all these computations one can easily verify that the Hessian of a smooth function $f$
has the following symmetric expression:
\begin{corollary}
$$
D^2_M f = \frac{1}{2} \Bigl[ D^2 f + D^2 \Phi  \cdot   D^2 \bigl[ f \circ (\nabla \Psi)\bigr] \circ \nabla \Phi \cdot  D^2 \Phi\Bigr].
$$
\end{corollary}

Writing down the expressions for the Hessians of $V$ and $W(\nabla \Phi)$ and applying  (\ref{ricci-V-W}) one can  easily get
\begin{corollary}
\label{Bak-Em}
The Bakry--{\'E}mery tensor 
$
\mbox{\rm{R}}_{\infty,\mu}$ has the following coordinate expression 
$$
\mbox{\rm{Ric}}_{ik} +  (D^2_M P)_{ik}
= \frac{1}{4}  g^{jl} g^{ms} \Bigl( \frac{\partial^3 \Phi}{\partial x_m \partial x_i \partial x_l} \Bigr) 
\Bigl(\frac{\partial^3 \Phi}{\partial x_s \partial x_j \partial x_k}\Bigr) + \frac{1}{2}  \frac{\partial^2 V}{\partial x_i \partial x_k} + \frac{1}{2} g_{il} g_{kj} \frac{\partial^2 W}{\partial x_l \partial x_j} \circ\nabla \Phi.
$$
\end{corollary}

Finally we obtain Theorem \ref{BE-est} and even more
\begin{theorem}
\label{g2}
The Bakry--{\'E}mery tensor $
\mbox{\rm{R}}_{\infty,\mu}$ satisfies
$$
\bigl(
\mbox{\rm{R}}_{\infty,\mu}\bigr)_{ij} \ge
 \frac{1}{2}  \frac{\partial^2 V}{\partial x_i \partial x_k} + \frac{1}{2} g_{il} g_{kj} \frac{\partial^2 W}{\partial x_l \partial x_j} \circ\nabla \Phi.
$$
One has
$$
\mbox{\rm{R}}_{\infty,\mu} \ge C \cdot g
$$
provided
$$
 (D^2 \Phi)^{-\frac{1}{2}} D^2 V (D^2 \Phi)^{-\frac{1}{2}} + (D^2 \Phi)^{\frac{1}{2}}  D^2 W(\nabla \Phi) (D^2 \Phi)^{\frac{1}{2}} \ge 2C \cdot \mbox{\rm{Id}}.
$$

In particular, if $V$ and $W$ are convex, when $
\mbox{\rm{R}}_{\infty,\mu} \ge 0$.
\end{theorem}

\section{Concentration and isoperimetric properties of $M$}

 We denote by $d$ the standard Euclidean distance and by $d_M$ the distance on the manifold $M$.

Almost everywhere below we deal with log-concave measures. Recall that the probability measure with Lebesgue density
$\mu = e^{-V} \ dx$ is called log-concave if $V$ is a convex function. Note that this function may take  value $+\infty$ outside of 
a convex set $A$. A function $\beta_i \in L^1(\mu)$ is called logarithmic derivative of $\mu$ along $x_i$ is the following identity holds
$$
\int \beta_i \varphi \ d \mu = \int \partial_{x_i} \varphi \ d \mu
$$ 
for every $\varphi \in C^{\infty}_{0}(\mathbb{R}^d)$.
If $V$ is regular enough, the only reasonable candidate for $\beta_i$ is $V_{x_i}$. If $V_{x_i}$ is well-defined and integrable with respect to $\mu$, then this is indeed the case. However, not every log-concave measure has a logarithmic derivative.
It is easy to check that the normalized Lebesgue on a compact convex set does not have any logarithmic derivative.

In this section we study concentration properties of the manifold $M$ under assumption  that the target measure
$\nu$ is compactly supported. 
They easily follows from the concentration properties of $\mu$  with the help of the following lemma.
A similar result see in \cite{EK} (Lemma 3.2).

\begin{lemma}
The following estimate holds
$$
d^2_M(x,y) \le  \langle \nabla \Phi(y) - \nabla \Phi(x), y - x \rangle \le d(x,y) \cdot d (\nabla \Phi(x), \nabla \Phi(y))
$$
\end{lemma}
\begin{proof}
Take two points $x, y$ and join them with the line $t \to  x + t v$, where $v=\frac{y-x}{d(x,y)}$. By the definition of the Riemannian distance one has
\begin{align*}
& d^2_M(x,y)  \le \Bigl( \int_0^{d(x,y)}  \sqrt{\langle D^2 \Phi(x+tv) v, v \rangle }\ dt \Bigr)^2
\le d(x,y) \int_0^{d(x,y)}  {\langle D^2 \Phi(x+tv) v, v \rangle }\ dt
\\& =  d(x,y) \langle \nabla \Phi(x+ tv), v \rangle |^{d(x,y)}_0
=  d(x,y) \langle \nabla \Phi(y)  - \nabla \Phi(x), v \rangle
\le  \langle \nabla \Phi(y) - \nabla \Phi(x), y - x \rangle
\\&
\le d(x,y) \cdot d(\nabla \Phi(y), \nabla \Phi(x)).
\end{align*}
\end{proof}

\begin{corollary}
Assume that $\nu$ has a bounded support $\mbox{\rm{diam}}(\rm{supp}(\nu)) = D$. Then
for every $A \subset \mathbb{R}^d$ one has
$$
\mu(x: d_M(x,A) \le h) \ge \mu\Bigl(x: d(x,A)  \le \frac{h^2}{D} \Bigr)
$$
\begin{proof}
According to the previous lemma
$$
\mu\bigl(x: d_M(x,A) \le h\bigr) \ge 
\mu\bigl(x: d(x,A) d(\nabla \Phi(x), \nabla \Phi(A)) \le h^2\bigr) \ge
\mu\bigl(x: d(x,A) D \le h^2\bigr).
$$
\end{proof}
\end{corollary}

In particular, one can estimate the concentration function of the metric-measure space $(M, d_M, \mu)$.
Recall that a function $\mathcal{K}_{\mu}$ is called concentration function for $\mu$ if it satisfies the following inequality:
$$
\mu(x: d(x,A) \le h) \ge 1- e^{-\mathcal{K_{\mu}}(h)}
$$
for every set $A$ with  $\mu(A) \ge \frac{1}{2}$ and any $h\ge 0$.

In particular, if  $\mbox{\rm{diam}}(\rm{supp}(\nu)) = D$ and $\mu$ is standard Gaussian, then
\begin{equation}
\label{concentrat}
\mu(x: d_M(x,A) \le h) \ge 1- e^{-\frac{1}{2}\bigl(\frac{h^2}{D}\bigr)^2}, \ \mu(A) \ge \frac{1}{2}.
\end{equation}

In the absence of the uniform bound  some concentration estimates are  still available.
\begin{corollary}
Let $\mathcal{K}_{\mu}$ and $\mathcal{K}_{\nu}$ be concentration functions of $\mu$ and $\nu$ respectively.
Then for every $h>0, t>0$ and $A \subset \mathbb{R}^d$ satisfying $\mu(A)\ge \frac{1}{2}$
$$
\mu\bigl(x: d_M(x,A) \le h\bigr) 
\ge
1- e^{-\mathcal{K_{\mu}}\bigl(\frac{h^2}{t}\bigr)} -  e^{-\mathcal{K_{\nu}}(t)}.
$$
\end{corollary}
\begin{proof}
\begin{align*}
\mu\bigl(x: d_M(x,A) \le h\bigr) & \ge 
\mu\bigl(x: d(x,A) d(\nabla \Phi(x), \nabla \Phi(A)) \le h^2\bigr) 
\\& \ge
\mu\bigl(x: d(\nabla \Phi(x), \nabla \Phi(A)) \le  t;  d(x,A) t \le h^2\bigr)
\\&
\ge \mu\bigl(x: d(x,A) t \le h^2\bigr)  - \mu(x: d(\nabla \Phi(x), \nabla \Phi(A)) > t)
\\&
= \mu\bigl(x: d(x,A) t \le h^2\bigr)  - \nu(y: d(y, \nabla \Phi(A)) > t)
\\&
\ge 
1- e^{-\mathcal{K_{\mu}}\bigl(\frac{h^2}{t}\bigr)} -  e^{-\mathcal{K_{\nu}}(t)}.
\end{align*}
The very last inequality follows from the definition of the concentration function
$$
\mu\bigl(x: d(x,A) t \le h^2\bigr)  \ge 1- e^{-\mathcal{K_{\mu}}\bigl(\frac{h^2}{t}\bigr)},
$$
$$
 - \nu(y: d(y, \nabla \Phi(A)) > t) =  \nu(y: d(y, \nabla \Phi(A)) \le  t) -1 
\ge -e^{-\mathcal{K_{\nu}}(t)}.
$$
\end{proof}

This means, in particular, that the  concentration function of our metric-measure space $M$
can be estimated by
\begin{equation}
\label{KM-est}
\mathcal{K}_M(h) \ge - \log \inf_{t>0} \bigl[ e^{-\mathcal{K}_{\mu}\bigl(\frac{h^2}{t}\bigr)} + e^{-\mathcal{K}_{\nu}(t)}  \bigr].
\end{equation}
Till the end of the section we deal with the log-concave target and source measures. It is well known (the most general statement
of this type has been obtained by Emanuel Milman \cite{Milman}) that concentration inequalities imply isoperimetric inequalities under assumption of
the positivity of the the Bakry--{\'E}mery tensor. More precisely
\begin{theorem}
\label{MilmTh}
{\bf [E. Milman]} Let $M$ be a smooth complete oriented connected Riemannian manifold equipped with the measure
$\mu = e^{-P} \ d {\rm vol}$. Assume that the corresponding Bakry--{\'E}mery tensor is nonnegative and the
concentration function $\mathcal{K}_M$ satisfies $\mathcal{K}_M(r) \ge \alpha(r)$, $\forall r \ge \alpha^{-1}(\log 2)$.
 Then the isoperimetric function $\mathcal{I}_{M}(t)$  of $M$ can be estimated from below by
$$
\min\Bigl(c_1 t   \gamma\bigl(\log \frac{1}{t}\bigr), c_2\Bigr), \ \ \gamma(x) = \frac{x}{\alpha^{-1}(x)},
$$
where the constant $c_1$ is universal and $c_2$ depend solely on $\alpha$.
\end{theorem}

We apply this theorem and the concentration inequalities obtained above to get 
 isoperimetric inequalities on $M$. We have to overcome on this way certain technical difficulty: eventual non-completeness
of  $M$. In the following lemma we establish sufficient conditions for a Hessian manifold to be complete, but these
conditions are not always fulfilled in the applications we consider.

\begin{lemma}
\label{suff-complete}
Assume that $V$ and $W$ are  smooth functions, defined on the whole $\mathbb{R}^d$, $C=\sup_{x} \| D^2 W(x)\| < \infty$ and there exists $c>0$ such that $D^2 V (x) \ge c \cdot \mbox{\rm{Id}}$ for every $x \in \mathbb{R}^d$. Then $M$ is unbounded and complete.
\end{lemma}
\begin{proof}
It is sufficient to show that there exists $\varepsilon>0$ such that $d_M(x,y) \ge \varepsilon d(x,y)$.
It follows from the Caffarelli-type estimate (see, for instance, \cite{Koles2011}) that the norm of $\|D^2 \Psi\|$, where $\Psi=\Phi^*$, is uniformly bounded
by $\sqrt{\frac{C}{c}}$. Since $\nabla \Phi$ and $\nabla \Psi$ are reciprocal and $D^2 \Phi = (D^2 \Psi)^{-1} \circ \nabla \Phi$,
one has $D^2 \Phi \ge \sqrt{\frac{C}{c}} \mbox{Id}$. Hence
$$
d_M(x,y) = \inf_{\gamma, \ \gamma(0)=x, \gamma(1)=y} \int_{\gamma} \sqrt{\langle D^2 \Phi(\gamma(s)) \dot{\gamma},\dot{\gamma} \rangle } \ ds 
\ge \sqrt[4]{\frac{C}{c}} \cdot d(x,y).
$$
\end{proof}

\begin{corollary}
Assume that $V$ and $W$ satisfy the assumptions of Lemma \ref{suff-complete}. Then the isoperimetric function
of $M$ satisfies the conclusion of Theorem \ref{MilmTh} with $\mathcal{K}_M$ satisfying (\ref{KM-est}).
\end{corollary}

The assumptions of Lemma \ref{suff-complete} are very restrictive. However, it is possible to prove some isoperimetric-type estimates in the situation when $M$ is  not complete.

\begin{proposition}
\label{prop-poin}
Let $\mu=\gamma$ be the standard Gaussian measure and $\nu$ be a log-concave measure with bounded support $\Omega$, where $\mbox{diam}(\Omega)=D$.
There exists an universal constant $c$ such that   $M$ satisfies the following Poincar{\'e} inequality
\begin{equation}
\label{Mpoin}
\int |\nabla_M f |^2 \ d \mu \ge \frac{c}{D}\int \Bigl( f - \int f \ d \mu \Bigr)^2 \ d \mu,
\end{equation}
where $\mbox{\rm{diam}}(\Omega) = D$
for every locally Lipschitz function $f : \mathbb{R}^d \to \mathbb{R}$ with a bounded (in the standard Euclidean metric) support.
\end{proposition}
\begin{proof}
Without loss of generality we assume that $0 \in \Omega \subset \{x: |x| \le D\}$.

{\bf Step 1}. Let us construct a sequence of smooth convex functions $W_n$ with the following properties:
\begin{enumerate}
\item the measures $\nu_n = e^{-W_n} \ dx$ converge weakly to $\nu$,
\item every $W_n$ has uniformly bounded second derivatives,
\item for every $n>0$ there exists a number $N(n)$ such that $D^2 W_k \ge n \cdot \mbox{Id}$ on $\{ x: |x| \ge 2D\}$ for $k > N(n)$. 
\end{enumerate}
This type of construction is quite standard and we omit here the details. We need to show that the corresponding sequence (or just a subsequence) of optimal transportations $\nabla \Phi_n$
converges in a sense to $\nabla \Phi$. In fact, we show  the following  a priori estimate:
$
\sup_{n} \int \|D^2 \Phi_n\|^2 \ d \gamma < \infty.
$
 Then it follows immediately from the compactness embedding theorem (applied to local Sobolev spaces)  and  convexity of $\Phi_n$ 
that there exists a subsequence (denoted again by the same index) such that 1)  $\nabla \Phi_n \to \nabla \Phi$ almost everywhere, 2) $\partial_{x_i x_j}  \Phi_n \to  \partial_{x_i x_j}  \Phi $
weakly in $L^2_{loc}(\mathbb{R}^d)$ and $L^2(\gamma)$ for every $i,j$.

Let us show that the desired estimate holds indeed. We apply the following a priori bound proved in \cite{Koles2010} (see section 9 below)
\begin{equation}
\label{vw}
\int |\nabla V|^2 e^{-V} dx \ge \int \mbox{Tr} \bigl[ D^2 \Phi \cdot D^2 W(\nabla \Phi) \cdot D^2 \Phi \bigr] \ e^{-V} \ dx,
\end{equation}
which holds for every sufficiently regular measures $e^{-V }\  dx$, $e^{-W} \ dx$ and the corresponding optimal transportation $\nabla \Phi$.
The direct application of this inequality gives, however, nothing, because the sequence $D^2 W_n$ is not supposed to be  uniformly bounded  from below by a positive matrix.
Let us do the following trick: we apply this inequality to the measures $\tilde{\mu}_n = e^{-P(|\nabla \Phi_n|^2/2)} \cdot \gamma$
and $\tilde{\nu}_n = e^{-P(x^2/2)} \cdot \nu_n$ with some convex function $P$ to be chosen later. Note that $\nabla \Phi_n$ pushes forward $\tilde{\mu}_n$ onto $\tilde{\nu}_n$, hence estimate (\ref{vw}) is applicable.
One obtains
\begin{align*}
\int | P'(|\nabla \Phi_n|^2/2)  &\cdot D^2 \Phi_n \nabla \Phi_n  + x|^2 e^{-P(|\nabla \Phi_n|^2/2)} \ d \gamma 
\\& \ge   \int \mbox{Tr} \bigl[ D^2 \Phi_n \cdot (D^2 W_n + D^2 [P(x2/2)] ) \circ \nabla \Phi_n \cdot {D^2 \Phi_n}\bigr] e^{-P(|\nabla \Phi_n|^2/2)}\ d \gamma.
\end{align*}
By the Cauchy-Bunyakovsky inequality 
for every couples of vectors $a,b$
 there exists $C_{\varepsilon}$ such that
$$
|a+b|^2 = |a|^2 + 2 \langle a, b \rangle + |b|^2 \le (1+\varepsilon) |a|^2 + C_{\varepsilon} |b|^2.
$$
Thus
\begin{align*}
\int | P'(|\nabla \Phi_n|^2/2)  & \cdot D^2 \Phi_n \nabla \Phi_n  + x|^2 e^{-P(|\nabla \Phi_n|^2/2)} \ d \gamma 
\le C_{\varepsilon} \int |x|^2  e^{-P(|\nabla \Phi_n|^2/2)} \ d \gamma  
\\& +  (1+\varepsilon) \int \bigl[ P'(|\nabla \Phi_n|^2/2) \bigr]^2  \cdot |D^2 \Phi_n \nabla \Phi_n|^2 e^{-P(|\nabla \Phi_n|^2/2)} \ d \gamma.
\end{align*}
Applying the identity
$$
|D^2 \Phi_n \nabla \Phi_n|^2
= \mbox{Tr} \Bigl( D^2 \Phi_n \cdot  \bigl( \nabla \Phi_n \otimes \nabla \Phi_n\bigr) \cdot  D^2 \Phi_n  \Bigr)
$$
one gets
\begin{align*}
C_{\varepsilon} & \int |x|^2   e^{-P(|\nabla \Phi_n|^2/2)} \ d \gamma \\& \ge \int \mbox{Tr} \bigl[ D^2 \Phi_n \cdot \Bigl[ D^2 W_n + D^2 [P(x^2/2)]
  -  (1+\varepsilon) (P'(x^2/2))^2  x \otimes x \Bigr] \circ \nabla \Phi_n  \cdot {D^2 \Phi_n}\bigr] e^{-P(|\nabla \Phi_n|^2/2)}\ d \gamma.
\end{align*}
We prove the desired estimate if we find a bounded $P$ such that 
\begin{equation}
\label{23.07}
D^2 W_n + P'(x^2/2) \cdot \mbox{Id} + \bigl[ P''(x^2/2) - (1 + \varepsilon) (P'(x^2/2))^2 \bigr]  x \oplus x   \ge c \cdot \mbox{Id}
\end{equation}
for some $c>0$.
To this end we choose for $P$ any smooth non-decreasing function such that $P(t)=\delta t$ for $t \le 2 D^2$  and $P(t) = 2 D^2 \delta +1 $ for $t \ge 3D^2$.
Choosing a sufficiently small $\delta$, sufficiently big $n$, and taking into account property 3) of $W_n$ we easily conclude that (\ref{23.07}) holds at least starting from some number $n_0$.
Indeed, for carefully chosen parameters  the part depending on $P$ is uniformly bounded from below for $t< 2 D^2$ and   exceeds  $- \sup_{n \ge n_0, |x| \ge 2 D} \|D^2 W_n(x)\|$ for
$t \ge 2 D^2$.  
Finally, using that $1 \le P \le 2 D^2 \delta +1$, we get
$$
C_{\varepsilon}  \int |x|^2   \ d \gamma \ge c e^{-(2 D^2 \delta +1)}\int \|D^2 \Phi \|^2_{\mathcal{HS}} \ d \gamma,
$$
where $\| \cdot \|_{\mathcal{HS}}$ is the Hilbert-Schmidt norm.
We note that  $\Phi$ is a smooth function as far as $W_n$ is smooth inside of the support $\Omega$ of the measure $\nu$. This follows from the regularity theory for the Monge--Amper{\`e} equation (see \cite{Koles2010} for details).

{\bf Step 2}. 
 The property 2) of approximating measures ensures that every metric-measure space $M_n = (\mathbb{R}^d, D^2 \Phi_n, \gamma)$ is complete (Lemma \ref{suff-complete}) and has a non-negative 
 Bakry--{\'E}mery tensor. Theorem \ref{MilmTh} is applicable and we conclude that every $M_n$ satisfies an
isoperimetric inequality defined by its concentration function $\mathcal{K}_{M_n}$. In particular, one has for every 
locally Lipschitz function $f$
\begin{equation}
\label{n-poin}
\int |\nabla_{M_n} f |^2 \ d \gamma \ge c \cdot C^2_n \int \Bigl( f - \int f \ d \mu \Bigr)^2 \ d \gamma,
\end{equation}
where $C_n$ is (any) constant satisfying $\mathcal{I}_{M_n}(t) \ge C_n t$, $0 \le t \le \frac{1}{2}$, $\mathcal{I}_{M_n}$
is the isoperimetric function of $M_n$, and $c$ is a universal constant.  It remains to prove that we get
(\ref{Mpoin}) in the limit.

We apply (\ref{KM-est}) to estimate $\mathcal{K}_{M_n}$. Choosing $t$ to be equal to $D$ in (\ref{KM-est}) one can easily see that the term 
$K_{\nu_n}(D)$ tends to zero. This follows from the weak convergence and the fact that $\nu$  is supported on the set $\Omega \subset \{x: |x|\le D\}$. This means
that the concentration function of the limiting space is estimated from below by $K_{\gamma}(h^2/D)$. Clearly, $K_M(h) \ge c \frac{h^4}{D^2}$ for some universal $c$.
Let us apply a simple rescaling argument. Consider  the new metric $ \tilde{d}_M = \frac{1}{\sqrt{D}} d_M$.
The metric-measure space $\tilde{M} = (\mathbb{R}^d, \tilde{d}_M, \gamma)$ still has a non-negative Bakry--{\'E}mery tensor
and its concentration function is bigger than $c h^4$. Hence its Poincar{\'e} constant can be estimated by some universal number.
Then it follows immediately that $C = \underline{lim}_n C_n \sim \frac{1}{\sqrt{D}}$.

Since $\Phi$ is  smooth, it is sufficient to prove (\ref{Mpoin}) for 
a function of the type $f = g(\nabla \Phi)$, where $g$ is locally Lipschitz such that the support of $g$ lies positive distance of
$\partial \Omega$.  
Let us apply (\ref{n-poin}) to $f_n = g(\nabla \Phi_n)$.
It remains to show that 
$$
\lim_n \int |\nabla_{M_n} f_n |^2 \ d \gamma = \lim \int |\nabla_M f |^2 \ d \gamma. 
$$
Note that $\lim_n \int |\nabla_{M_n} f_n |^2 \ d \gamma  = \sum_{i,j} \int \partial_{{x_i x_j}} \Phi_n g_{x_i}(\nabla \Phi_n) g_{x_j}(\nabla \Phi_n) \ d \gamma.$  One has $\nabla \Phi_n \to \nabla \Phi$ almost everywhere. The desired convergence
follows immediately from the fact that
$\partial_{{x_i x_j}} \Phi_n \to \partial_{{x_i x_j}} \Phi$ weakly in $L^2(\gamma)$.
The proof is complete.
\end{proof}


\section{Applications: bounds for integral operator norms and reverse H{\"o}lder inequalities}

\subsection{A variance estimate for the operator norm $\| D^2 \Phi\|$}

In this section we will apply some results and techniques developed in \cite{Koles2011}.
Denote by 
$$
\Lambda = \| D^2 \Phi\|
$$
the operator norm of $D^2 \Phi$.
The following inequality for the operator norm has been
proved in   \cite{Koles2011}. 
Assume that $V$ and $W$ are twice continuously differentiable functions. 
Then
\begin{equation}
\label{V-Lambda}
\sup_{v: |v|=1} V_{vv}
 \ge
\inf_{v: |v|=1} W_{vv}(\nabla \Phi) \cdot \Lambda^2
- L_{\Phi} \Lambda 
+ \frac{\langle \nabla_M \Lambda, \nabla_M \Lambda \rangle_M}{\Lambda},
\end{equation}
where $ L_{\Phi} \Lambda $ is understood in the distributional sense.
To be more precise, for every nonnegative smooth compactly supported function $\eta$ one has
$$
\int
(\sup_{v: |v|=1} V_{vv}) \eta \ d \mu
\ge \int  (\inf_{v: |v|=1} W_{vv}(\nabla \Phi))\Lambda^2  \eta \ d \mu
+ \int \langle \nabla_M \Lambda, \nabla_M \eta \rangle_M \ d \mu
+ \int \frac{\langle \nabla_M \Lambda, \nabla_M \Lambda \rangle_M}{\Lambda} \eta \ d \mu.
$$
If there exists a global  smooth field of the eigenvectors $v$ corresponding to the largest eigenvalue $\Lambda$, this identity
follows immediately from 
Lemma 7.1 of  \cite{Koles2011} and the relations
$$(\partial_v D^2 \Phi) v = \nabla \Lambda, \ \ \mbox{Tr} \bigl[ \partial_v D^2 \Phi  \cdot Dv \cdot (D^2 \Phi)^{-1}\bigr] \ge 0$$ 
(see Theorem 7.3 \cite{Koles2011}).
For the full justification of this formula in the  general case see the proof of Theorem 7.3 in \cite{Koles2011}.

Approximating function $f(\Lambda)$ by smooth compactly supported functions, one obtains the following theorem.

\begin{theorem}
\label{integr-eigenvalue}
For every non-negative differentiable function $f$ the following inequality holds
\begin{align*}
\int v_{+}(x)   f(\Lambda) \ d \mu  
\ge
 \int w_{-}(\nabla \Phi) \Lambda^2  f(\Lambda) \ d \mu + \int  \Bigl( f'(\Lambda)  + \frac{f(\Lambda)}{\Lambda} \Bigr)\langle  \nabla_M \Lambda, \nabla_M \Lambda  \rangle_{M}  \ d \mu,
\end{align*}
where 
$$v_{+}(x) = \sup_{e: |e|=1} V_{ee}(x), \ \ \  w_{-}(x)  = \inf_{e: |e|=1} W_{ee}(x).$$
\end{theorem}

In what follows we obtain some bounds on the  operator norms of $D^2 \Phi$ in the case when the target measure is a log-concave measure on a bounded convex set $\Omega$
and the source measure $\mu = \gamma$ is Gaussian.
It is an open problem whether $\int \| D^2 \Phi\| \ d \gamma \le C \mbox{diam}(\Omega) $ for some universal $C$. We obtain below some related results.

\begin{theorem}
Let $\mu=\gamma$ be the standard Gaussian measures and $\nu$ be a log-concave measure  on a bounded  convex set 
$\Omega$. There exists an universal constant $c$ such that
$$
\int \Lambda \ d \mu  - \Bigl( \int \sqrt{\Lambda} \  d \mu \Bigr)^2 \le c D,
$$
where $\Lambda$ is
the operator norm of $D^2 \Phi$,
and
$\mbox{\rm{diam}} (\Omega) = D$.
\end{theorem}
\begin{proof}
The formal proof can be obtained by applying Theorem 
\ref{integr-eigenvalue}  to  $f=1$. One obtains inequality 
\begin{equation}
\label{L-M}
1 \ge 4 \int \langle \nabla_M \sqrt{\Lambda}, \nabla_M \sqrt{\Lambda} \rangle_M \ d\gamma.
\end{equation}
Then the result follows from the Poincar{\'e} inequality for $M$ proved in (\ref{Mpoin}).
Note, however, that (\ref{Mpoin}) is proved for compactly supported functions only, which is not the case with $\Lambda$. 
To avoid this difficulty we approximate $\nu$ (in the same way as in the Proposition \ref{prop-poin})  by smooth
 log-concave measures $\nu_n = e^{-W_n} \ dx$ with uniformly bounded second derivatives. We get (\ref{L-M}) for $M_n$ (note that it is applicable to locally Lipschitz functions, see the proof Proposition \ref{prop-poin}).
By  the Poincar{\'e} inequality we get boundedness of $c_n\Bigl( \int \Lambda_n \ d \mu  - \Bigl( \int \sqrt{\Lambda_n} \  d \mu\Bigr)^2 \Bigr),$
where $c_n$ are Poincar{\'e} constants of $(\mathbb{R}^d, D^2 \Phi_n,\gamma)$. In the same way as in Proposition
\ref{prop-poin} we get in the limit that $C \Bigl( \int \Lambda \ d \mu  - \Bigl( \int \sqrt{\Lambda}   d \mu \Bigr)^2 \Bigr) $ is bounded by $1$, where $C \sim \frac{1}{D}$ .
\end{proof}

\subsection{Reverse H{\"o}lder inequalities}

We conclude this section by a brief discussion of the reverse H{\"o}lder inequalities for $\Lambda$. Application of Theorem  
\ref{integr-eigenvalue} to
the powers of $\Lambda$
gives a kind of reverse H{\"o}lder inequalities for $p \ge 0$:
$$
\int \Lambda^{p+1} \ d \gamma \le \bigl( \int   \Lambda^{\frac{p+1}{2}} \ d \gamma  \bigr)^2 + c_p D \int \Lambda^p  \ d \gamma.
$$
It is also possible to obtain certain  reverse H{\"o}lder inequalities  with {\it universal} constants. To this end let
us note that
$$
\langle  \nabla_M \Lambda, \nabla_M \Lambda  \rangle_{M} = \langle  (D^2 \Phi)^{-1} \nabla \Lambda, \nabla \Lambda  \rangle \ge \frac{\langle \nabla \Lambda, \nabla \Lambda  \rangle}{\Lambda}.
$$
We come back to the standard Gaussian measure and Euclidean structure.
One gets
$$
\int \Lambda^p \ d\gamma \ge (p+1) \int |\nabla \Lambda|^2  \Lambda^{p-2} \ d\gamma.
$$
Apply the Poincar{\'e} inequality for the standard Gaussian measure:
$$
\int \Lambda^p \ d\gamma \ge \frac{4(p+1)}{p^2} \bigl( \int \Lambda^p \ d\gamma  - ( \int \Lambda^{p/2} \ d\gamma)^2  \bigr).
$$
One gets that for $p$ satisfying $4(p+1) > p^2$ there exists $C_p$ such that
$$
\int \Lambda^p \ d\gamma \le C_p \bigl( \int \Lambda^{p/2} \ d \gamma \bigr)^{2}.
$$

\section{Computations of higher order and Calabi-type estimates}

In this section we give a list of formulas which can be useful for further investigations. 
In particular, in this section we generalize the classical Calabi's computations from \cite{Calabi} (see also \cite{CNS}).
We don't give any specific applications of this, but we believe that this may be an important technical tool for many other problems.
For instance, the Calabi estimates were in the heart of the Yau's approach to the complex Monge-Amp{\`e}re equation in his solution of the Calabi's problem.

\subsection{Basic quantities} 

Unlike the Calabi's approach we deal with the diffusion operator $L_{\Phi}$
and don't use neither covariant derivatives nor geometrical identities. Nevertheless, we use the language
of Riemannian geometry which is very well adapted for this computations.

The functions $V$, $W$ and $\Phi$ are supposed to be smooth. We use the standard summation convention. All the computation are made in the fixed (global) chart and $g= D^2 \Phi$.
$f_m$ denotes the partial  derivative of $f$ with respect to $x_m$ and we set:
$$
V_{i} = V_{x_i}, \ V_{ij} = V_{x_i x_j}, \ldots
$$
$$
\Phi_{i} = V_{x_i}, \ \Phi_{ij} = V_{x_i x_j}, \ldots, \Phi^{i}_{j}  = g^{ik} \Phi_{kj}, \ldots
$$
It is convenient to set
$$
W^{i} = W_{x_i}(\nabla \Phi), \ W^{ij} = W_{x_i x_j}(\nabla \Phi), \ldots.
$$

First we note that the relation  $V_{x_i} = - L_{\Phi} \Phi_{x_i}$ can be rewritten as follows:
\begin{equation}
\label{18.07}
-V_i + W_i = g^{jk} \Phi_{ijk}.
\end{equation}

It is convenient to apply the following relations in the computation
$$
(\Phi_{ij})_{k} = - \Phi^{ij}_{k}.
$$

\begin{equation}
\label{diff'2}
  L_{\Phi} g_{ij}  = -V_{ij} + W_{ij} +  \Phi^{ab}_{i} \Phi_{abj}.
\end{equation}
\begin{equation}
\label{diff'3}
L_{\Phi} g^{ij} =
V^{ij} - W^{ij}
+  \Phi^{iab} \Phi^{j}_{ab}.
\end{equation}

\begin{proof}
One has
$$L_{\Phi} f 
= g^{ij} f_{ij} - W^i f_{i}.
$$ 
From the relation $g^{ij}g_{jk} = \delta_{ik}$ one gets
\begin{equation}
\label{g-1.1}
(g^{is})_m = - g^{ij} g^{ks} \Phi_{jkm} = -\Phi^{is}_m.
\end{equation}
\begin{equation}
\label{g-1.2}
(g^{is})_{mr} = - g^{ij} g^{ks} \Phi_{jkmr} + 2 g^{ij} g^{kb} g^{as} \Phi_{bar} \Phi_{jkm} = - \Phi^{is}_{mr} + 2 \Phi^{ia}_m \Phi^{s}_{ar} .
\end{equation}
Differentiating  $-V_{x_i} = L_{\Phi} \Phi_{x_i}$ one obtains
$$
  L_{\Phi} g_{ij}  = -V_{x_i x_j} + W_{x_k x_s}(\nabla \Phi)  g_{ik} g_{js} + g^{ak} g^{bl} \Phi_{abi} \Phi_{klj}.
$$
This is exactly (\ref{diff'2}).

We get from  (\ref{g-1.2}) that
\begin{align*}
g^{mr} (g^{is})_{mr} & = - g^{mr} g^{ij} g^{ks} \Phi_{jkmr} + 2 g^{mr} g^{ij} g^{kb} g^{as} \Phi_{bar} \Phi_{jkm} 
\end{align*}
In the other hand, (\ref{diff'2}) implies
\begin{align*}
- g^{mr} \Phi_{jkmr}
= V_{x_j x_k}  - W_{x_r x_s}(\nabla \Phi)  g_{rj} g_{ks} - g^{ar} g^{bp} \Phi_{abj} \Phi_{rpk}  - W_{x_r}(\nabla \Phi) \Phi_{jkr}.
\end{align*}
Hence
\begin{align*}
g^{mr} (g^{is})_{mr} & =
g^{ij} g^{ks}  V_{x_j x_k} +  g^{ij} g^{ks}g^{ar} g^{bp} \Phi_{abj} \Phi_{rpk}
- g^{ij} g^{ks}  g_{rj} g_{kp} W_{x_r x_p}(\nabla \Phi) 
- g^{ij} g^{ks} W_{x_r}(\nabla \Phi) \Phi_{jkr}.
\end{align*}
Then (\ref{g-1.1}) implies (\ref{diff'3}).
\end{proof}

The following lemma is obtained by direct computations and we omit the proof here.
One gets (\ref{diff'4}) by differentiating (\ref{diff'2}). The relation (\ref{diff'5}) follows from (\ref{diff'4}), (\ref{g-1.2}),
and (\ref{diff'3}) by the Leibnitz rule.

\begin{lemma}
One has
\begin{align}
\label{diff'4}
L_{\Phi} \Phi_{ijk} = &
 \bigl(\Phi_{abi} \Phi^{ab}_{jk} + \Phi_{abj} \Phi^{ab}_{ik} +  \Phi_{abk} \Phi^{ab}_{ij} \bigr) - 2 \Phi^{a}_{bi} \Phi^{b}_{cj} \Phi^{c}_{ak}
\\& \nonumber
- V_{ijk} + W_{ijk}
+ \bigl(  W^{s}_{i} \Phi_{sjk} + W^{s}_{j} \Phi_{sik} + W^{s}_{k} \Phi_{sij} \bigr).
\end{align}
\begin{align}
\label{diff'5}
L_{\Phi} \Phi^{ijk} = &
- \bigl(\Phi^{abi} \Phi_{ab}^{jk} + \Phi^{abj} \Phi_{ab}^{ik} +  \Phi^{abk} \Phi_{ab}^{ij} \bigr) + 
4 \Phi^{ia}_{b} \Phi^{jb}_{c} \Phi^{kc}_{a}
\\& \nonumber
- V^{ijk}  + W^{ijk}
+ \bigl(  V^{is}\Phi_{s}^{jk} + V^{sj} \Phi_{s}^{ik} + V^{sk} \Phi_{s}^{ij} \bigr)
\\& \nonumber
+ \Phi^a_{xy} \bigl[   \Phi^{ixy} \Phi^{jk}_a + \Phi^{jxy} \Phi^{ik}_a  + \Phi^{kxy} \Phi^{ij}_a \bigr].
\end{align}
\end{lemma}

\subsection{Calabi-type estimates}

We are interested in the quantity
$$
L_{\Phi} \bigl(  \Phi^{abc} \Phi_{abc}\bigr).
$$
The computation of this quantity was the main technical point of the Calabi's approach. Function $R = \frac{1}{4} \Phi^{abc} \Phi_{abc}$
coincides with the scalar curvature for the case of constant $V$ and $W$ (this corresponds to the optimal transportation of the Lebesgue measure on  convex sets). Calabi has shown that in this case
$$
\Delta_M R \ge C(d) R^2.
$$
Thus $R$ is superharmonic and this fact can be used to control the growth of $R$. In fact,  $L_{\Phi} \bigl(  \Phi^{abc} \Phi_{abc}\bigr)$
controls the fourth-order derivatives too.

We compute this quantity by applying the Leibnitz rule and the Lemmata from the previous section. 
Let us  omit the lengthy computations and state the final result.
\begin{proposition}
\label{Calabi-calcul}
One has
$$
L_{\Phi} \bigl(  \Phi^{abc} \Phi_{abc}\bigr)
= \mbox{\rm{I}} + \mbox{\rm{II}} + \mbox{\rm{III}} ,
$$
where 
$$
\mbox{\rm{I}} = 3 V_{ab} \Phi^{aij} \Phi^{b}_{ij} + 3 W^{ab} \Phi^{ij}_a \Phi_{ijb},
\ \ \
\mbox{\rm{II}} = -2 V^{abc}  \Phi_{abc}  + 
2 W^{abc}\Phi_{abc},
$$
\
$$
\mbox{\rm{III}} =  3 \Phi^{abc} \Phi^{def}  \Phi_{aef} \Phi_{dbc} + 2 \Phi^{abc} \Phi^{d}_{ae} \Phi^{e}_{bf}  \Phi^{f}_{cd}
-6  \Phi^{abcd} \Phi^{e}_{ab} \Phi_{cde}
+ 2  \Phi^{abcd} \Phi_{abcd}.
$$
\end{proposition}

A remarkable observation which goes back to Calabi is the following

\begin{proposition} 
There exists a universal constant $C>0$ such that
$$
{\rm{III}} \ge  C \Bigl[  \Phi^{abc} \Phi^{def}  \Phi_{aef} \Phi_{dbc}  + \Phi^{abcd} \Phi_{abcd}  \Bigr] \ge 0.
$$
\end{proposition}
\begin{proof}
First we note that the value of the expression $\Phi^{abc} \Phi_{abc}$ 
is invariant with respect to  any orthogonal coordinate  change. 
Indeed,
$$
\Phi^{abc} \Phi_{abc} = \mbox{Tr} \Bigl[ (D^2 \Phi)^{-1} \cdot B  \Bigr],
$$
where
$$
B_{ij} = \mbox{Tr} \Bigl[(D^2 \Phi)^{-1} \cdot \bigl( D^2 \Phi \bigr)_{e_i} \cdot  (D^2 \Phi)^{-1} \cdot \bigl( D^2 \Phi \bigr)_{e_j} \Bigr]
$$
The direct computations give immediately that $\tilde{B} = O^* B O$, where $x \to Ox$ is an orthogonal transformation and $\tilde{B}$ is the corresponding matrix in the new basis. This implies the invariance.

Fix a point $x$ and  choose an orthogonal   basis in such a way that $D^2 \Phi(x)$ is diagonal. Denote by
$\mu^i$, $1  \le i \le n$ the eigenvalues of $(D^{2} \Phi)^{-1}$. Then 
\begin{align*}
\rm{III}  &= \mu^{a} \mu^b \mu^c \mu^d \Bigl( 3 \Phi_{abc} \Phi_{def} \Phi_{aef} \Phi_{dbc} \   \mu^e \mu^f+
2 \Phi_{abc} \Phi_{ade} \Phi_{bef} \Phi_{cdf}  \ \mu^e \mu^f
\\& - 6 \Phi_{abcd} \Phi_{abe} \Phi_{cde} \ \mu^e + 2 \Phi^2_{abcd}
\Bigr).
\end{align*}
Note that
$$
\mu^{a} \mu^b \mu^c \mu^d  \mu^e \mu^f \Phi_{abc} \Phi_{def} \Phi_{aef} \Phi_{dbc} 
=
 \mu^b \mu^c \mu^e   \mu^f
\Bigl( \mu^a  \Phi_{abc}  \Phi_{aef}  \Bigr)^2.
$$
Rearranging the indices one can rewrite $\rm{III}$ in the following way:
\begin{align*}
& \mu^a \mu^b \mu^c \mu^d 
\Bigl[  \frac{3}{2} ( \mu^f \Phi_{abf} \Phi_{cdf})^2 +  \frac{3}{2} ( \mu^f \Phi_{acf} \Phi_{bdf})^2 
+
 2  \Bigl[ \mu^e \Phi_{abe} \Phi_{cde} \Bigr]  \Bigl[ \mu^f \Phi_{acf} \Phi_{bdf} \Bigr]
\\&
- 3 \Phi_{abcd} ( \sqrt{\mu^e}\ \Phi_{abe}) (\sqrt{\mu^e} \ \Phi_{cde})  - 3 \Phi_{abcd} ( \sqrt{\mu^f}\ \Phi_{acf}) (\sqrt{\mu^f} \ \Phi_{bdf})
+ 2 \Phi^2_{abcd} \Bigr].
\end{align*}
Set: 
$$
 Y_{ab} = (\sqrt{\mu^1} \Phi_{ab1}, \cdots, \sqrt{\mu^n} \Phi_{abn}).
$$
Then
$$
\rm{III} =  \mu^a \mu^b \mu^c \mu^d  Q_{abcd},
$$
where
$$
Q_{abcd} = 2 \Phi^2_{abcd} - 3 \Phi_{abcd} \langle Y_{ab}, Y_{cd} \rangle   - 3 \Phi_{abcd} \langle Y_{ac}, Y_{bd} \rangle + 2    \langle Y_{ab}, Y_{cd} \rangle  \langle Y_{ac}, Y_{bd} \rangle 
+ \frac{3}{2}  \langle Y_{ab}, Y_{cd} \rangle^2 + \frac{3}{2}  \langle Y_{ac}, Y_{bd} \rangle^2. 
$$
for every fixed $a,b,c,d$. It is easy to check that the function
$$
2x^2 - 3xy - 3xz + 2yz + \frac{3}{2} y^2 + \frac{3}{2} z^2
$$
is non-negative.
Hence
$$
Q_{abcd} \ge C \bigl( \Phi^2_{abcd} +  \langle Y_{ab}, Y_{cd} \rangle^2 +  \langle Y_{ac}, Y_{bd} \rangle^2 \bigr).
$$
The proof is complete.
\end{proof}

\begin{corollary} (Calabi-type estimates). 
Note that
$$
\Phi^{abc} \Phi^{def}  \Phi_{aef} \Phi_{dbc} = \mbox{\rm{Tr}} \Bigl[(D^2 \Phi)^{-1} \cdot B^2 \cdot (D^2 \Phi)^{-1} \Bigr],
$$
where $B$ is defined as above. Now it is an easy 
exercise to recover the following  Calabi's estimate:
$$
\Phi^{abc} \Phi^{def}  \Phi_{aef} \Phi_{dbc} = \mbox{\rm{Tr}} \Bigl[(D^2 \Phi)^{-1} \cdot B^2 \cdot (D^2 \Phi)^{-1} \Bigr]
\ge \frac{1}{d} \Bigl( {\mbox{\rm{Tr}}(B \cdot (D^2 \Phi)^{-1} )} \Bigr)^2 = \frac{1}{d} (\Phi^{abc} \Phi_{abc} )^2.
$$
Hence, if $\mu$ and $\nu$ are both normalized Lebesgue measure on convex sets, then
$$
L_{\Phi} (\Phi^{abc} \Phi_{abc} ) \ge \frac{c}{d}( \Phi^{abc} \Phi_{abc})^2.
$$
for some universal $c>0$ .
\end{corollary} 

\section{A remark on K{\"a}hler manifolds and convex sets}

Let us assume that $\mu$ is a log-concave measure and $\nu$ is the normalized Lebesgue measure on a convex set $\Omega$.
It was shown in \cite{Klartag} that for any Lipschitz function $f: \Omega \to \mathbb{R}$ with $\int_\Omega f \ dx =0$ under certain additional technical assumption on $\Phi$ (regularity at infinity) one has
the following Poncar{\'e}-type inequality
\begin{equation}
\label{Kl-ineq}
\int_{\Omega} f^2 \ dx \le \int_{\Omega} Q_{\Phi, x}(\nabla f) \ dx,
\end{equation}
where 
$$
Q_{\Phi, x} (v) = \sup \Bigl\{ 4 g_{ij}(\nabla \Phi^*) v^i v^j; \ v \in \mathbb{R}^n,  \ Q^{*}_{\Phi,x}(v) \le 1 \Bigr\},
$$
and
$$
Q^{*}_{\Phi,x}(v) = v^i v^j (g^{lm} g^{kp} \Phi_{jmk} \Phi_{ilp}) \circ \nabla \Phi^*.
$$
This inequality implies some thin-shell estimates on the simplex. An important point in the proof
from \cite{Klartag} was the embedding of the initial space into a  toric K{\"a}hler manifold. This K{\"a}hler manifold
admits a nonnegative Ricci tensor. Finally, the results follow from the Bochner's identity. 
It was pointed out to the author by Bo'az Klartag that  inequality (\ref{Kl-ineq}) 
can be generalized to the case when $\nu$ is any log-concave measure if instead of K{\"a}hler structure one applies 
the metric-measure space studied in this paper and Theorem \ref{BE-est}.

For applications of  Hessian structures in statistics see  \cite{Shima} and the related references.

\section{Hessian manifolds as $CD(K,N)$-spaces and diameter bounds}

We recall that a smooth $d$-dimensional manifold equipped with a measure $\mu = e^{-P} \ d \mbox{vol}$, $P \in C^2(M)$
is called $CD(K,N)$-space if the {\it modified} Bakry--{\'E}mery tensor
$$
\mbox{\rm{R}}_{N,\mu} = \mbox{\rm{Ric}} + D^2_M P - \frac{1}{N-d} \nabla_M P \otimes \nabla_M P, \ \ N >d
$$
satisfies
$$
\mbox{\rm{R}}_{N,\mu} \ge K.
$$
In the case $N=\infty$ the tensor $\mbox{\rm{R}}_{\infty,\mu}$ coincides with the Bakry--{\'E}mery tensor.

Let us show that our metric-measure space is a $CD(K,N)$-space.
\begin{lemma}
$$
\bigl(\mbox{\rm{R}}_{\infty,\mu}\bigr)_{ii} \ge \frac{1}{2} \bigl( V_{ii}+ W_{ii} \bigr) + \frac{1}{4d} (- V_i + W_i)^2.
$$
\end{lemma}
\begin{proof}
The result follows from  Corollary \ref{Bak-Em}. It is helpful to apply the following relation
$$
 \Phi^{j}_{mi} \Phi^{m}_{jk}  = g^{jl} g^{ms} \Phi_{x_m x_i x_l} \Phi_{x_s x_j x_k} 
= \mbox{Tr}\bigl[ (D^2 \Phi)^{-1} D^2 \Phi_{x_i} (D^2 \Phi)^{-1} D^2 \Phi_{x_k} \bigr]. 
$$
By the Cauchy inequality and (\ref{18.07})
$$
\mbox{Tr} \bigl[(D^2 \Phi)^{-1} D^2 \Phi_{x_i} (D^2 \Phi)^{-1} D^2 \Phi_{x_i}\bigr]
\ge \frac{1}{d} ( \mbox{Tr} (D^2 \Phi)^{-1} D^2 \Phi_{x_i} )^2 = \frac{1}{d} (-V_i + W_i)^2.
$$
\end{proof}

\begin{theorem}
Assume that $\mu = e^{-V}  dx$ is a log-concave measure and $\nu$
is the normalized Lebesgue measure on a convex set $\Omega$. Then $M=(\mathbb{R}^d,g,\mu)$ is $CD(0,2d)$-space.

If, in addition, $$D^2 V \ge C \cdot \mbox{{\rm Id}}$$ with $C>0$, then $(\mathbb{R}^d,g,\mu)$ is $CD\bigl(\frac{C}{m},2d\bigr)$-space,
where $m = \sup_{x \in \mbox{\rm supp}(\mu)} \|D^2 \Phi\|$.
\end{theorem}
\begin{proof}
Recall that $\mu = e^{-P} \ d \mbox{vol}_M$, where $P = \frac{1}{2} (V +W(\nabla \Phi))$. Since $W$ is constant on $\mbox{supp}(\nu)$,
one gets from the previous lemma that
$$
\mbox{\rm{R}}_{\infty,\mu} \ge \frac{C}{2 m} g + \frac{1}{d} \nabla_M P \otimes \nabla_M P.
$$
\end{proof}

\begin{remark}
It was pointed out to the author by Bo'az Klartag that this result can be obtained by embedding $M$ onto the corresponding toric $2d$-dimensional K{\"a}hler manifold and applying 
computations in $\mathbb{C}^d$ (which are easier than in the real case). This gives a geometric interpretation of the constant $2d$ appearing in the $CD(K,N)$-condition.
\end{remark}

It is known that the $CD(K,N)$-spaces satisfy  the Myer's and the Bishop-Gromov comparison theorem
(see \cite{Villani2}). But this holds in general for complete manifolds only, which is not always the case with $M$.
Indeed, one can easily verify that if $\mu$ is the standard Gaussian measure and $\nu$ is the Lebesgue measure on
a simple  set (ball or cube), then $M$ is bounded and {\bf not} complete.
Nevertheless, some classical results can be easily verified for the case of geodesically convex manifolds, which are manifold admitting the following property:
every two points $x, y \in M$  can be joined by a shortest geodesic curve $\gamma(t) : t \to \exp_x(tv), \ t \in [0,\mbox{\rm{dist}}(x,y)], \ v \in TM_x, \ |v|=1$.

In particular, the standard proof of the Bishop-Gromov comparison theorem for volume growth of balls can be easily generalized to geodesic convex spaces.
Applying this fact  we get  immediately the following result.

\begin{corollary}
\label{Bishop}
Assume that $\mu = e^{-V}  dx$ is a log-concave measure and $\nu$
is the Lebesgue measure on a convex set $\Omega$. Assume, in addition, that $M$ is geodesically convex.
Then 
$$
r \to \frac{\mu(\{x : d_{M}(x,x_0) \le r\})}{r^{2d}}
$$
is a non-increasing function.
\end{corollary}

It is known (see \cite{Villani2} for precise statements and for the  references) that concentration inequalities together with non-negativity of the Ricci tensor imply boundedness of the manifold.
Here we apply concentration arguments to our space $M$.

\begin{theorem}
\label{diam-bound}
Let $\mu = \gamma$ be the standard Gaussian measure and $\nu$ is the Lebesgue measure  on a convex set 
with diameter $D$. Assume, in addition, that $M$ is geodesically convex. There exists a universal constant $C>0$ such that
$$
\mbox{\rm{diam}}(M) \le C \sqrt[4]{d}  \sqrt{D}.
$$
\end{theorem}
\begin{proof}
We follow the arguments from \cite{Ledoux} Theorem 7.4.
Apply the concentration inequality (\ref{concentrat}).
\begin{equation}
\label{arm}
\mu(x: d_M(x,A) \le h) \ge 1- e^{\frac{h^4}{2D^2}},  \ \ \mu(A) \ge \frac{1}{2}.
\end{equation}
Choose a ball  $\{x: d_{M}(0,x) \le r \} = B_{r}(0)$ on $M$ with $r= \frac{\mbox{\rm{diam(M)}}}{8}$.
Take $z$ at distance $3r$ from the origin. One has $B_r(0) \subset B_{4r}(z).$ If $\mu(B_r(0))\ge \frac{1}{2}$ set $A=B_r(0)$ and apply (\ref{arm}).  By Corollary \ref{Bishop}
$$
\mu(B_{r}(z)) \ge \mu(B_{4r}(z)) \cdot \Bigl( \frac{1}{4} \Bigr)^{2d}.
$$
Note that $B_{r}(z)$ is included in the complement of $A_{r}$, hence 
\begin{equation}
\label{10.02.12-1}
e^{-\frac{r^4}{2D^2}} \ge \frac{1}{2}\cdot \Bigl( \frac{1}{4} \Bigr)^{2d}.
\end{equation}
If $\mu(B_r(0))\le \frac{1}{2}$ set: $A=B^c_r(0)$. One has
$$
\mu(B_{\frac{r}{2}}(0)) \ge \mu(B_{8r}(0)) \frac{1}{16^d}.
$$ 
The ball $B_{\frac{r}{2}}(0)$
is included in the complement of $A_{\frac{r}{2}}$. One obtains
\begin{equation}
\label{10.02.12-2}
e^{-\frac{r^4}{32 D^2}} \ge  \frac{1}{16^d}.
\end{equation}
Inequalities (\ref{10.02.12-1}) and (\ref{10.02.12-2}) imply the desired bound.
\end{proof}

\begin{remark}
There are some reasons to believe that the
 assumption of geodesic convexity of $M$   can be omitted, but we were not able to prove this.
\end{remark}

\end{document}